\documentclass[12pt, draft]{amsart}
\usepackage{latex_base}
\usepackage{macros}

\author{Janin Heuer}

\address{Janin Heuer, Technische Universit\"at Braunschweig, Institut f\"ur Analysis und Algebra, AG Algebra, Universit\"atsplatz 2, 38106 Braunschweig,
 Germany\medskip}

\email{janin.heuer@tu-braunschweig.de}

\author{Ngoc Mai Tran}

\address{Ngoc Mai Tran, Department of Mathematics, The University of Texas at Austin, Speedway 2515 Stop C1200, Austin TX 78712, USA\medskip}

\email{ntran@math.utexas.edu}

\author{Timo de Wolff}

\address{Timo de Wolff, Technische Universit\"at Braunschweig, Institut f\"ur Analysis und Algebra, AG Algebra, Universit\"atsplatz 2, 38106 Braunschweig,
 Germany\medskip}

\email{t.de-wolff@tu-braunschweig.de}



\title[A Generalized Muirhead Inequality and Symmetric SONCs]{A Generalized Muirhead Inequality and Symmetric Sums of Nonnegative Circuits}

\begin{document}

\begin{abstract}
Circuit polynomials are a certificate of nonnegativity for real polynomials, which can be derived via a generalization of the classical inequality of arithmetic and geometric means.
In this article, we show that similarly nonnegativity of symmetric real polynomials can be certified via a generalization of the classical Muirhead inequality.
Moreover, we show that a nonnegative symmetric polynomial admits a decomposition into sums of nonnegative circuit polynomials if and only if it satisfies said generalized Muirhead condition.
The latter re-proves a result by Moustrou, Naumann, Riener, Theobald, and Verdure for the case of the symmetric group in a shortened and more elementary way.
\end{abstract}

\maketitle

\section{Introduction}

Nonnegative real polynomials are a fundamental object in real algebraic geometry. 
Deciding nonnegativity is a highly challenging problem. 
Thus, one is interested in finding conditions which imply nonnegativity but can be checked effectively. 
The first such \struc{certificate of nonnegativity} is the \struc{sum of squares (SOS)} certificate, which has been studied as early as the 19th century \cite{Hilbert:Seminal}.
During the last 20 to 25 years, the field (re-)gained substantial momentum due to its close connection to polynomial optimization \cite{Blekherman:Parrilo:Thomas,Lasserre:BookMomentsApplications,Lasserre:IntroductionPolynomialandSemiAlgebraicOptimization,Laurent:Survey}, including successful applications in various different areas such as control theory \cite{Parrilo:Thesis,PapachristodoulouPrajna:LyapunovSOS} and nonlinear partial differential equations \cite{Mevissen:Solving,Marx:Moment}.

A recent approach to certifying polynomial nonnegativity is based on \struc{sums of nonnegative circuits (SONC)}. 
A special case of this certificate was introduced by Reznick in \cite{Reznick:AGI}; the formulation for the general polynomial case was given by Iliman and the third author in \cite{Iliman:deWolff:Circuits}; see also \cite{Craciun:Koeppl:Pantea:GlobalInjectivity,Fidalgo:Kovacec} for independent, but related approaches.
Circuit polynomials have received rising attention in recent years as a criterion for deciding nonnegativity as decompositions can be computed via either convex optimization \cite{Chandrasekaran:Shah:SAGE-REP,Iliman:deWolff:FirstGP} or via symbolic computations in applications such as chemical reaction networks \cite{Feliu:Kaihnsa:Yueruek:deWolff:Multistationarity,Feliu:Kaihnsa:Yueruek:deWolff:Multistationarity:2}.  

A circuit polynomial's support set is minimally affine dependent. 
This, combined with some further conditions given in \cref{definition:Circuit}, ensures that deciding nonnegativity of a circuit polynomial is reduced to computing an invariant called \struc{circuit number} via solving a system of linear equations. 
This circuit number can be inferred from the classical \struc{inequality of arithmetic and geometric means (AM-GM inequality)}. 

For sums of squares (SOS), if the polynomial of interest is \emph{symmetric}, one could exploit symmetry in order to reduce the size of the semidefinite optimization problems corresponding to the SOS certificates. 
A path-leading contribution in this context was given by Gaterman and Parrilo \cite{Gaterman:Parrilo}, followed by many others, for example \cite{Raymond:Saunderson:Singh:Thomas,Riener:Theobald:Andren:Lasserre}.
An analogous result for SONC was recently obtained in \cite{Moustrou:Naumann:Riener:Theobald:Verdure:Symmetry}. 
That is, if a polynomial has underlying symmetry, then the size of the optimization problem that needs to be solved to obtain a SONC certificate can be reduced. 
The authors prove this result using abstract algebraic methods. 

In this paper, we show that a SONC certificate for symmetric polynomials can also be proven as a consequence of another classical inequality -- the \struc{Muirhead inequality} \cite{Muirhead}. 
This is a symmetric generalization of the AM-GM inequality and thus a canonical tool for studying symmetric SONCs.
We provide an overview on Muirhead's inequality and how it can be used in order to decide polynomial nonnegativity in \cref{sec:Muirhead}. 
In \cref{sec:GenMuirhead} we prove a generalized version of Muirhead's inequality which extends the classical inequality from simply comparing symmetric sums of monomials to also allowing certain scalar coefficients.
In \cref{sec:SymmetricSONC} we then use the generalized Muirhead inequality to derive in \cref{corollary:SymmetricSONC} conditions for nonnegativity of real symmetric polynomials.
These Muirhead conditions can be re-formulated in a way that involves circuit numbers for SONCs mentioned above.
In consequence, we show in \cref{theorem:MuirheadSONCEquivalence} that symmetric polynomials satisfy these conditions if and only if they are SONC.
This result reproves a theorem by Moustrou, Naumann, Riener, Theobald, and Verdure presented in \cite{Moustrou:Naumann:Riener:Theobald:Verdure:Symmetry} for the symmetric group, showcasing it as a natural consequence of a classical inequality, and avoiding further algebraic methods.
Finally, in \cref{cor:SymmetricClosure} we provide a reformulation of our results interpreting it as a description of the symmetric closure of the SONC cone for an arbitrary support.

\section*{Acknowledgments}
We thank Thorsten Theobald for his helpful comments.
JH and TdW were supported by the DFG grant WO 2206/1-1. NT was partially supported by NSF-DMS grant \#2113468, and the NSF IFML \#2019844 award to the University of Texas at Austin.
\section{Preliminaries}

We begin by fixing some basic notation. Bold notations are vectors, for example, we write $\struc{\xb}$ for $(x_1,\ldots,x_n) \in \R^n$. We denote the set of nonnegative real numbers by $\struc{\R_{\ge 0}}=\{ x \in \R \ : \ x \ge 0\}$, the set of real vectors indexed by elements in some set $A$ by $\struc{\R^A}$, and the set $\set{1, 2, \ldots, m}$ by $\struc{[m]}$ for some $m \in \N$.
We write the \struc{convex hull} of some set $S \subseteq \R^n$ as $\struc{\conv(S)}$, and the \struc{relative interior} of $\conv(S)$ as $\struc{\relint(\conv(S))}$.

\subsection{Sums of Nonnegative Circuits}

Let $\struc{\R[\xb]}$ be the ring of real, multivariate polynomials in $\xb \in \R^n$. 
Then every polynomial in $\R[\xb]$ is of the form
\begin{align*}
	f(\xb) \ = \ \sum_{\alpb \in A} c_{\alpb} \xb^{\alpb}
\end{align*} 
for some set of exponents $A \subset \N^n$, and with \struc{coefficients} $c_{\alpb} \in \R$.
We call the individual summand of $f$ \struc{monomials}.
The set $\struc{\supp(f)} = \set{ \alpb \ : \ \alpb \in A, \ c_{\alpb} \ne 0}$ is called the \struc{support (set)} of $f$. 
For every $f \in \R[\xb]$ we can assign the corresponding \struc{Newton polytope} given by 
\begin{align*}
	\struc{\New(f)} \ = \ \conv(\supp(f)) \subset \R^n \ .
\end{align*}

Deciding whether a function $f \in \R[\xb]$ is nonnegative is NP-hard in general; see for example \cite{Laurent:Survey}.
Hence, one is interested in finding \struc{certificates of nonnegativity}.
These are conditions on real multivariate polynomials which imply nonnegativity, hold for a reasonably large class of polynomials, and are easier to test than nonnegativity itself.
Here we focus on the \struc{SONC certificate}, which states that a polynomial is nonnegative if it can be decomposed into a \struc{sum of nonnegative circuit (SONC) polynomials}.

\begin{definition}[Circuit Polynomial]
	A polynomial $f \in \R[\xb]$ is called a \struc{circuit polynomial} if it is of the form
	\begin{align*}
		f(\xb) \ = \ \sum_{\alpb \in A^+} c_{\alpb} \xb^{\alpb} + c_{\betab} \xb^{\betab}
	\end{align*} 
	with support $A = A^+ \cup \set{\betab} \subset \N^n$, such that 
	\begin{enumerate}
		\item the coefficients satisfy $c_{\alpb} \in \R_{>0}$ for all $\alpb \in A^+$ and $c_{\betab} \in \R$, 
		\item $\New(f)$ is a simplex with even vertices, which are given by $A^+$, and 
		\item the exponent $\betab$ is in the strict interior of $\New(f)$. 
	\end{enumerate}
	Note that this implies that $A^+ \subset (2\N)^n$.
	\label{definition:Circuit}
\end{definition}

Recall that in matroid theory $A \subset \N^n$ is called a \struc{circuit} if $A$ is minimally affinely dependent, that is, $A$ is affinely dependent but any proper subset of $A$ is is affinely independent \cite[pp.9]{Oxley:MatroidTheory}.
As all circuits considered in \cref{definition:Circuit} and in what follows are given by the vertices (the $\alpb \in A^+$) of a simplex with one point in the interior (the $\betab$), we refer to the $\alpb$'s also as \struc{outer points} and to $\betab$ as \struc{inner point} of the corresponding circuit.
The notion of circuit polynomials in the sense of \cref{definition:Circuit} was first introduced by Iliman and the second author in \cite{Iliman:deWolff:Circuits}, building on work by Reznick in \cite{Reznick:AGI}.
A key property of circuit polynomials is that since $\New(f)$ is a simplex, we can uniquely represent the interior point $\betab$ as a convex combination of the vertices $\alpb \in A^+$.
That is, there exist \emph{unique} barycentric coordinates $\Vector{\lam} \in (0,1)^{A^+}$ with $\sum_{\alpb \in A^+} \lam_{\alpb} = 1$ such that 
\begin{align}
	\betab \ = \ \sum_{\alpb \in A^+} \lam_{\alpb} \alpb \ .
\end{align}
The nonnegativity of circuit functions can be easily decided by an invariant called the \struc{circuit number}, which is defined as
\begin{align*}
	\struc{\Theta_f} \ = \ \prod_{\alpb \in A^+} \left(\frac{c_{\alpb}}{\lam_{\alpb}}\right)^{\lam_{\alpb}} \ .
\end{align*}

\begin{theorem}[{\cite[Theorem 1.1]{Iliman:deWolff:Circuits}}]
	Let $f=\sum_{\alpb \in A^+} c_{\alpb} \xb^{\alpb} + c_{\betab} \xb^{\betab}$ be a circuit polynomial.
	Then $f$ is nonnegative if and only if
	\begin{align}
		\begin{array}{c c c}
			|c_{\betab}| \ \le \ \Theta_f,& \text{ or } & c_{\betab} \ge 0 \ .
		\end{array}
	\label{eq:NonnegCircuit}
	\end{align}
	\label{theorem:CircuitsNonnegativity}
\end{theorem}

This approach to deciding polynomial nonnegativity can be motivated as a consequence of the classical \struc{inequality of arithmetic and geometric means (AM-GM)}
\begin{align}
	\frac{1}{m}\sum_{i=1}^{m} t_i \ \ge \ \sqrt[m]{\prod_{i=1}^{m} t_i}
	\label{eq:AMGM}
\end{align}
for nonnegative real numbers $t_1, \ldots, t_m$.

\begin{example}
	Consider \cref{eq:AMGM} and let $m = 3$, and $t_1 = x^4y^2, t_2 = x^2y^4, t_3 = 1$.
	This yields
	\begin{align*}
		\frac{x^4y^2 + x^2y^4 + 1}{3} \ \ge \ \sqrt[3]{x^6y^6} \ ,
	\end{align*}
	or, equivalently,
	\begin{align}
		x^4y^2 + x^2y^4 - 3x^2y^2 + 1 \ \ge \ 0 \ .
		\label{eq:Motzkin}
	\end{align} 
	This proves that the \struc{Motzkin polynomial}, given on the left hand side of the inequality \cref{eq:Motzkin}, is nonnegative.

\end{example}

With this, we formally define the cone of all sums of nonnegative circuit polynomials.

\begin{definition}[SONC Cone]
	The \struc{SONC cone} $\struc{\SONC}$ is the cone of all polynomials which have a representation as sums of nonnegative circuit polynomials or monomial squares. 
	If we fix a support set $A$, then we indicate this by writing $\struc{\ASONC}$.
	In some contexts it is useful to additionally fix the signs of all terms with exponents in the support set $A \subset \N^n$.
	Then we decompose the support $A$ into disjoint sets $\struc{A^+} \subset (2\N)^n$, $ \struc{A^-} \subset \N^n$.
	That means, we intersect the configuration space $\R^A$ with one particular orthant given by the chosen sign pattern. 
	In this case, we denote the corresponding SONC (sub-)cone by $\struc{\signedSONC}$.
\label{definition:SONCCone}
\end{definition}



For the case of polynomials on the positive orthant, or exponential sums, another equivalent certificate of nonnegativity has been studied.
Basis for this certificate are polynomials or exponential sums with at most one negative term, but without the additional conditions of circuit polynomials. 
These type of functions are called \struc{AGE functions}, and were first introduced by Chandrasekaran and Shah in \cite{Chandrasekaran:Shah:SAGE-REP}.
It has been shown in \cite{Wang:nonnegative,Murray:Chandrasekaran:Wierman:NewtonPolytopes} that every SONC polynomial can equivalently be studied as a \struc{sum of AGE (SAGE) functions}, i.e., the SONC and SAGE cone are the same object. 
For more on SAGEs, see e.g. \cite{Chandrasekaran:Shah:SAGE-REP,Murray:Chandrasekaran:Wierman:NewtonPolytopes,Murray:Chandrasekaran:Wierman:SigOptREP}. \\

In the SAGE language, the special case of \emph{symmetric} SONC polynomials has recently been explored. 
In \cite{Moustrou:Naumann:Riener:Theobald:Verdure:Symmetry}, the authors present a way to exploit symmetries in order to reduce the number of variables and constraints involved in finding a decomposition into AGE functions, see \cite[Theorem 4.1]{Moustrou:Naumann:Riener:Theobald:Verdure:Symmetry} for the exact result.
It turns out that we can derive a similar result using little more than known results about SONC polynomials, and the classical \struc{Muirhead's inequality}.

\subsection{Muirhead's Inequality}
\label{sec:Muirhead}

Similarly to the AM-GM inequality, Muirhead's inequality involves a comparison of sums of multivariate monomials. 
In fact, Muirhead's inequality is a generalization of the AM-GM inequality to \emph{symmetric sums}; see \cref{example:MuirheadAMGM}.
It was first published by Muirhead in 1903 in \cite{Muirhead}.
It can be derived from the AM-GM inequality, but independent proofs are known, for example, using the Birkhoff-von Neumann Theorem. For more information, see Hardy, Littlewood, and P\'olya \cite{Hardy:Littlewood:Polya:Inequalities}. \\

Let $\alpb \in \N^n$ be a lattice point and let $\struc{S_n}$ denote the symmetric group of permutations of $n$ elements. 
To be able to refer to fixed elements in $S_n$, we enumerate them as $S_n = \{ \sig_1, \dots, \sig_{n!} \}$. 
Furthermore, we define
\begin{align*}
	\struc{\Sigma(\alpb)} \ = \ \{ (\alp_{\sig(1)}, \dots,\alp_{\sig(n)}) \ : \ \sig \in S_n\} \subset \N^n \ .
\end{align*}
That is, $\Sigma(\alpb)$ contains all vectors which can be obtained by permuting the entries of $\alpb$.
Thus, if we use the notation $\struc{\sig(\alpb)} = \left(\alp_{\sig(1)},\dots,\alp_{\sig(n)}\right) \in \N^n$ for every $\sig \in S_n$ and every lattice point $\alpb \in \N^n$, then we have 
\begin{align}
	\Sigma(\alpb) \ = \ \{ \sig_1(\alpb), \dots, \sig_{n!}(\alpb) \} \ .
	\label{eq:SetRepresentation}
\end{align}

\begin{theorem}[Muirhead's Inequality]
	Let $\betab \in \conv(\Sigma(\alpb))$. 
	Then
	\begin{align}
		\sum_{\sig \in S_n} x_{1}^{\beta_{\sig(1)}} \cdots x_{n}^{\beta_{\sig(n)}} \ \le \ \sum_{\sig \in S_n} x_{1}^{\alp_{\sig(1)}} \cdots x_{n}^{\alp_{\sig(n)}}
	\label{eq:Muirhead}
	\end{align}	
	for nonnegative real numbers $x_1, \ldots, x_n$.
	\label{theorem:Muirhead}
\end{theorem}
\begin{example}
	Let $\alpb = (n, 0, \ldots, 0)$ and $\betab = (1, 1, \ldots, 1)$ be vectors in $\N^n$. 
	Then $\Sigma(\alpb)$ is the scaled standard simplex in $\N^n$ and $\betab \in \conv(\Sigma(\alpb))$.
	Furthermore, it holds that $\sig(\betab) = \betab$ for every $\sig \in S_n$ and there are exactly $(n-1)!$ permutations $\tau \in S_n$ that satisfy $\tau(\alpb) = \sig(\alpb)$ for any arbitrary fixed $\sig \in S_n$.
	Muirhead's inequality now becomes
	\begin{align*}
		n! \left( x_{1} \cdots x_{n}\right) \ \le \ (n-1)! (x_1^n + \cdots + x_n^n)
	\end{align*}
	for $\xb \in \R^n_{\ge 0}$, showing that it is indeed a generalization of the AM-GM inequality \cref{eq:AMGM}.
	\label{example:MuirheadAMGM}
\end{example}
\begin{example}[{\cite[pp.186]{Steele:CauchySchwarzMasterClass}}]
	Consider the case $n = 3$, and let $\betab = \smallMatrix{2\\ 3\\ 0}$ and $\alpb = \smallMatrix{1\\ 4\\ 0}$.
	Then we have 
	\begin{align*}
		\betab \ = \ \Matrix{2\\ 3\\ 0} \ = \ \frac{2}{3} \Matrix{1\\ 4\\ 0} + \frac{1}{3} \Matrix{4\\ 1\\ 0} \ \in \ \conv(\Sigma(\alpb)) \ .
	\end{align*}
	The left hand side of \cref{eq:Muirhead} now reads
	\begin{align*}
		\sum_{\sig \in S_3} x_{1}^{\betab_{\sig(1)}} x_{2}^{\betab_{\sig(2)}} x_{3}^{\betab_{\sig(3)}} 
		\ = \ x_1^2x_2^3 + x_1^2x_3^3 + x_1^3x_2^2 + x_2^2x_3^3 + x_1^3x_3^2 + x_2^3x_3^2 \ ,
	\end{align*}
	and the right hand side becomes
	\begin{align*}
		\sum_{\sig \in S_3} x_{1}^{\alpb_{\sig(1)}} x_{2}^{\alpb_{\sig(2)}} x_{3}^{\alpb_{\sig(3)}}
		\ = \ x_1x_2^4 + x_1x_3^4 + x_1^4x_2 + x_2x_3^4 + x_1^4x_3 + x_2^4x_3 \ .
	\end{align*}
	Thus, Muirhead's inequality tells us that the polynomial
	\begin{align*}
		f(x_1, x_2, x_3) \ = \ &x_1x_2^4 + x_1x_3^4 + x_1^4x_2 + x_2x_3^4 + x_1^4x_3 + x_2^4x_3 \\
		&-x_1^2x_2^3 - x_1^2x_3^3 - x_1^3x_2^2 - x_2^2x_3^3 - x_1^3x_3^2 - x_2^3x_3^2
	\end{align*}
	is nonnegative for all $\xb \in \R^3_{\ge 0}$.
	
	We can verify this statement by using the SONC approach.
	For this, we first make a simple variable transformation and let $x_i = y_i^2$ for all $i \in [3]$, so that instead of considering a polynomial over the positive orthant we can treat $f$ as a multivariate polynomial in $\yb \in \R^3$.
	Then we can write 
	\begin{align*}
		f(y_1, y_2, y_3) \ &= \ y_1^2y_2^8 + y_1^2y_3^8 + y_1^8y_2^2 + y_2^2y_3^8 + y_1^8y_3^2 + y_2^8y_3^2 \\
		& \qquad -y_1^4y_2^6 - y_1^4y_3^6 - y_1^6y_2^4 - y_2^4y_3^6 - y_1^6y_3^4 - y_2^6y_3^4 \\
		&= \ \left(\frac{2}{3}y_2^2y_3^8 - y_2^4y_3^6 +\frac{1}{3}y_2^8y_3^2\right) 
		+ \left(\frac{1}{3}y_2^2y_3^8 - y_2^6y_3^4 + \frac{2}{3}y_2^8y_3^2\right) \\
		& \qquad + \left(\frac{2}{3}y_1^2y_3^8 - y_1^4y_3^6 +\frac{1}{3}y_1^8y_3^2\right)
		+ \left(\frac{1}{3}y_1^2y_3^8 - y_1^6y_3^4 + \frac{2}{3}y_1^8y_3^2 \right) \\
		& \qquad + \left(\frac{2}{3}y_1^2y_2^8 - y_1^4y_2^6 +\frac{1}{3}y_1^8y_2^2 \right)
		+ \left(\frac{1}{3}y_1^2y_2^8 - y_1^6y_2^4 + \frac{2}{3}y_1^8y_2^2 \right) \ .
	\end{align*}
One can check that each of the polynomials in this decomposition are circuit polynomials which satisfy the condition \cref{eq:NonnegCircuit} for nonnegativity.
\end{example}

\section{A Generalized Muirhead Inequality}
\label{sec:GenMuirhead}

We now present a generalized version of Muirhead's inequality.
As one would expect, several generalizations of Muirhead's inequality were proven during the 120 years since its publication; see e.g. \cite{Cuttler:Greene:Skandera,Kato:Kosuda:Tokushige} as two recent examples. However, to the best of our knowledge, the generalization stated in \cref{lemma:GeneralizedMuirhead} appears to be new.

Let $\alpb \in \N^n$, $n \ge 3$, and consider $\Sigma(\alpb)$ as represented in \cref{eq:SetRepresentation}. 
Let $\betab$ be a lattice point in $\conv(\Sigma(\alpb))$. 
By Carath\'eodory's theorem there exist $n+1$ elements $\sig_1(\alpb), \dots, \sig_{n+1}(\alpb)$ in $\Sigma(\alpb)$ such that 
\begin{align}
	\begin{array}{c c c c c}
		\betab \ = \ \sum\limits_{j = 1}^{n+1} \lam_j \sig_j(\alpb), & \text{ where } & \Vector{\lam} \ \in \ (0,1)^{n+1} & \text{ and } & \sum\limits_{j = 1}^{n+1} \lam_j \ = \ 1 \ .
	\end{array}
\label{eq:1} 
\end{align}
Therefore, for every $\tau \in S_n$ it holds that
\begin{align}
	\begin{array}{c c c c c}
		\tau(\betab) \ = \ \sum\limits_{j = 1}^{n+1} \lam_j \tau(\sig_j(\alpb)), & \text{ where } & \Vector{\lam} \ \in \ (0,1)^{n+1} & \text{ and } & \sum\limits_{j = 1}^{n+1} \lam_j \ = \ 1 \ .
	\end{array}
	\label{eq:2} 
\end{align}
Now we can show the following generalized version of Muirhead's inequality, using the short notation $\struc{\xb^{\tau(\alpb)}} = x_{1}^{\alp_{\tau(1)}} \cdots x_{n}^{\alp_{\tau(n)}}$.

\begin{lemma}[Generalized Muirhead Inequality]
	Let $\alpb \in \N^n$ for $n \ge 3$ and $\betab \in \relint(\conv(\Sigma(\alpb)))$. 
	Consider an arbitrary collection of $n+1$ nonnegative real numbers $b_{j} \in \R_{\ge 0}$ for $j \in [n+1]$ corresponding to the barycentric coordinates $\lam_j$ as in \cref{eq:1}. 
	Then
	\begin{align}
		\sum_{\tau \in S_n} \prod_{j = 1}^{n+1} b_j^{\lam_j} \xb^{\tau(\betab)} \ \le \ \sum_{\tau \in S_n} \sum_{j=1}^{n+1} b_j \lam_j \xb^{\tau(\alpb)}
		\label{eq:GeneralizedMuirhead}
	\end{align}
	holds for arbitrary $\xb \in \R_{\ge 0}^n$.
\label{lemma:GeneralizedMuirhead}
\end{lemma}

First, note that, in general, the inequality depends on our choice of the convex combination of the initial $\betab$; see \cref{eq:1}.
However, \textit{every} convex combination of $\betab$ in terms of the $\sigma_j(\alpb)$ can be chosen, and the lemma remains valid.
Second, observe that this statement is indeed a generalization of Muirhead's inequality, as choosing all $b_{\tau_j}$ equal to $1$ leads to the classical Muirhead inequality \cref{eq:Muirhead}, in which case the inequality does not depend on the choice of the convex combination, i.e., the $\lam_j$, anymore. 

\begin{proof}[Proof of \cref{lemma:GeneralizedMuirhead}]
	We combine the AM-GM inequality with the classical Muirhead inequality. 
	Assume that the conditions of \cref{lemma:GeneralizedMuirhead} are satisfied. 
	That is, let $\alpb \in \N^n$ for $n \ge 3$ and $\betab \in \relint(\conv(\Sigma(\alpb)))$ and let $b_{j} \in \R_{\ge 0}$ for $j \in [n+1]$, and assume that \cref{eq:1} holds. 
	Then we have 
	\begin{align*}
		\sum_{\tau \in S_n} \prod_{j = 1}^{n+1} b_j^{\lam_j} \xb^{\tau(\betab)} \ = \ \prod_{j = 1}^{n+1} b_j^{\lam_j} \sum_{\tau \in S_n} \xb^{\tau(\betab)} \ \le \ \prod_{j = 1}^{n+1} b_j^{\lam_j} \sum_{\tau \in S_n} \xb^{\tau(\alpb)} \ \le \ \sum_{j=1}^{n+1} b_j \lam_j \sum_{\tau \in S_n} \xb^{\tau(\alpb)} \ ,
	\end{align*}
	where the first inequality is given by the classical Muirhead inequality and the second one by the AM-GM inequality. 
	This proves the claim.
\end{proof}

\section{Symmetric Sums of Nonnegative Circuits}
\label{sec:SymmetricSONC}

We can use the above proof of our generalized version of Muirhead's inequality to infer a criterion for containment in the SONC cone which exploits symmetries. 
That means that we take the vectors $\alpb_i \in \Sigma(\alpb), \betab \in \N^n$ in the generalized Muirhead inequality to be the exponents of a circuit polynomial. 
At first glance, this seems to be a big restriction to the allowed support sets. But it turns out that the nonnegativity of any circuit polynomial can be reduced to nonnegativity of a circuit polynomials whose exponents satisfy these conditions.

\begin{lemma}
	Let $f(\xb) = \sum_{\alpb \in A^+} c_{\alpb} \xb^{\alpb} + c_{\betab} \xb^{\betab}$ be a circuit polynomial.
	Then $f$ is nonnegative if and only if for any arbitrary $\tilde{\alpb} \in (2\N)^n$ the polynomial
	\begin{align*}
		g(\xb) \ = \ \sum_{j = 1}^{n+1} \xb^{\sigma_j(\tilde\alpb)} + \left( c_{\betab} \cdot 
		\frac{ 
			n+1
			} 
			{\prod_{\alpb \in A^+} 
		\left(\frac{c_{\alpb}}{\lambda_{\alpb}}\right)^{\lambda_{\alpb}}} \right) \xb^{\tilde\betab} \ ,
	\end{align*}
	where $\tilde{\betab} = \sum_{j = 1}^{n+1} \frac{1}{n+1} \sigma_j(\tilde\alpb) - \sum_{\alpb \in A^+} \lambda_{\alpb} \alpb + \betab$ and the $\lam_{\alpb}$ are the barycentric coordinates of $\betab$ with respect to the $\alpb \in A^+$, is nonnegative.
\label{lemma:SupportReduction}
\end{lemma}

\begin{proof}
	By \cref{theorem:CircuitsNonnegativity}, the circuit polynomial $f$ is nonnegative if and only if $|c_{\betab}| \le \prod_{\alpb \in A^+} \left(\frac{c_{\alpb}}{\lambda_{\alpb}}\right)^{\lambda_{\alpb}}$, or if it is a sum of monomial squares.
	Since the latter case is trivial, assume that $c_{\betab} < 0$.
	Fix an arbitrary $\tilde\alpb \in (2\N)^n$.
	Then $\betab = \sum_{\alpb \in A^+} \lam_{\alpb} \alpb$ is equivalent to 
	\begin{align*}
		\sum_{j = 1}^{n+1} \frac{1}{n+1} \sigma_j(\tilde\alpb) \ = \ \sum_{j = 1}^{n+1} \frac{1}{n+1} \sigma_j(\tilde\alpb) - \sum_{\alpb \in A^+} \lambda_{\alpb} \alpb + \betab \ = \ \tilde\betab \ .
	\end{align*}
	With this, $\tilde{\betab}$ is contained in the relative interior of $\conv{\Sigma(\tilde{\alpb})}$ with barycentric coordinates $\left(\frac{1}{n+1}, \ldots, \frac{1}{n+1}\right)$ with respect to $\sigma_1(\tilde{\alpb}), \ldots, \sigma_{n+1}(\tilde{\alpb})$.
	Furthermore, we can trivially rewrite $|c_{\betab}| \le \prod_{\alpb \in A^+} \left(\frac{c_{\alpb}}{\lambda_{\alpb}}\right)^{\lambda_{\alpb}}$ as 
	\begin{align*}
		|c_{\betab}| \cdot 
		\frac{ 
			n + 1
			} 
			{\prod_{\alpb \in A^+} 
		\left(\frac{c_{\alpb}}{\lambda_{\alpb}}\right)^{\lambda_{\alpb}}} 
		\ \le \ 
		n + 1 \ .
	\end{align*}

	We can assume without loss of generality that $\tilde\betab \in \N^n$. 
	Since by \cref{theorem:CircuitsNonnegativity} 
	\begin{align*}
		g(\xb) \ = \ \sum_{j = 1}^{n+1} \xb^{\sigma_j(\tilde\alpb)} + \left( c_{\betab} \cdot 
		\frac{ 
			n+1
			} 
			{\prod_{\alpb \in A^+} 
		\left(\frac{c_{\alpb}}{\lambda_{\alpb}}\right)^{\lambda_{\alpb}}} \right) \xb^{\tilde\betab} \ ,
	\end{align*}
	is nonnegative if and only if 
	\begin{align*}
		\left( \vert c_{\betab} \vert \cdot 
		\frac{ 
			n+1
			} 
			{\prod_{\alpb \in A^+} 
		\left(\frac{c_{\alpb}}{\lambda_{\alpb}}\right)^{\lambda_{\alpb}}} \right) \ \le \ \prod_{j = 1}^{n+1} \left(n+1\right)^{\frac{1}{n+1}} \ = \ n+1 \ ,
	\end{align*}
	it follows by our above argumentation that the nonnegativity of $f$ and $g$ is equivalent.
\end{proof} 

The following result shows that the generalized Muirhead inequality \cref{lemma:GeneralizedMuirhead} yields a certificate of nonnegativity for symmetric polynomials. 
We introduce the notation
\begin{align*}
	\struc{\Sigma(A)} \ = \ \set{\sigma(\Vector{\gamma}) \ : \ \sigma \in S_n, \Vector{\gamma} \in A}
\end{align*}
for the symmetric closure of a set $A \subset \R^n$. 

Furthermore, we denote by $\struc{\cC(A^+, \betab)}$ the set of all circuits with inner point $\betab$ over a given signed support set $A^+ \cup \set{\betab} \subset \N^n$. 
We refer to elements of $\cC(A^+, \betab)$ by writing them as a tuple of the form $(C^+, \betab)$, where $C^+ \subseteq A^+$ and $\betab \in \relint(\conv(C^+))$.

\begin{corollary}
	Let $A = A^+ \cup A^- \subset \N^n$ and let 
	\begin{align*}
		f(\xb) \ = \ \sum_{\sig \in S_n} \sum_{\betab \in A^-} \sum_{(C^+, \betab) \in \cC(A^+, \betab)} f^{\sig}_{(C^+, \betab)}(\xb)
	\end{align*}
	be a symmetric polynomial supported on $\Sigma(A)$, such that the
	\begin{align}
		f^{\sig}_{(C^+, \betab)}(\xb) \ = \ \sum_{\alpb \in C^+} c_{\alpb}^{(C^+, \betab)} \xb^{\sig(\alpb)} + c_{\betab}^{(C^+, \betab)} \xb^{\sig(\betab)}
	\label{eq:SymmetricSum}
	\end{align}
	are circuit polynomials, where $c_{\betab}^{(C^+, \betab)} \le 0$.\footnote{Note that, in general, for fixed $\betab$ by far not all circuits in $\cC(A^+, \betab)$ are needed to represent $f(\xb)$; see for example \cite{Forsgaard:deWolff:BoundarySONCcone}. 
	In order to keep the notation as simple as possible, we summarize over all circuits and allow for redundant circuits the corresponding circuit polynomials $f^{\sig}_{(C^+, \betab)}(\xb)$ to be the zero polynomial.}

	Then $f(\xb)$ is nonnegative if for every $(C^+, \betab) \in \cC(A^+, \betab)$ with $\betab \in A^-$ one of the following equivalent conditions holds:
	\begin{enumerate}
		\item$
			\sum\limits_{\sig \in S_n} \left| c_{\betab}^{(C^+, \betab)} \right| \xb^{\sig(\betab)} \ \le \ \sum\limits_{\sig \in S_n} \prod\limits_{\alpb \in C^+} \left(\frac{c_{\alpb}^{(C^+, \betab)}}{\lam_{\alpb}^{(C^+, \betab)}}\right)^{\lam_{\alpb}^{(C^+, \betab)}} \xb^{\sig(\betab)}$
			\label{eq:MuirheadCondition}
		\item $\left| c_{\betab}^{(C^+, \betab)} \right| \ \le \ \prod\limits_{\alpb \in C^+} \left(\frac{c_{\alpb}^{(C^+, \betab)}}{\lam_{\alpb}^{(C^+, \betab)}}\right)^{\lam_{\alpb}^{(C^+, \betab)}}$ \ .
		\label{eq:CircuitCondition}
	\end{enumerate}	
	In both conditions, $\Vector{\lam}^{(C^+, \betab)}$ is the unique vector of barycentric coordinates of $\betab$ with respect to $C^+$.
\label{corollary:SymmetricSONC}
\end{corollary}

Condition (\cref{eq:MuirheadCondition}) in the statement above immediately implies that the generalized Muirhead inequality holds for the positive real numbers $\set{\frac{c_{\alpb}^{(C^+, \betab)}}{\lam_{\alpb}^{(C^+, \betab)}}}_{\alpb \in C^+}$. 
Condition (\cref{eq:CircuitCondition}) equivalently restates this in terms of a circuit number condition that exploits symmetry in the sense that the nonnegativity of the \cref{eq:SymmetricSum} is independent of $\sigma$ and thus only needs to be checked for one permutation in $S_n$.

\begin{proof}[Proof of \cref{corollary:SymmetricSONC}]
	The equivalence of (\cref{eq:MuirheadCondition}) and (\cref{eq:CircuitCondition}) is clear by term inspection.

	For the remainder of the proof, we first note that we can assume without loss of generality that the support sets of the $f^{\sig}_{(C^+, \betab)}(\xb)$ satisfy that $C^+$ is generated by a single vector $\tilde{\alpb} \in (2\N)^n$, i.e., that $C^+ = \Sigma(\tilde{\alpb})$ by \cref{lemma:SupportReduction}.
	If (\cref{eq:MuirheadCondition}) holds, then the nonnegativity of 
	\begin{align*}
		f(\xb) = \sum_{\betab \in A^-} \sum_{(C^+, \betab) \in \cC(A^+, \betab)} \left( \sum_{\sig \in S_n} f^{\sig}_{(C^+, \betab)}(\xb) \right)
	\end{align*}
	follows directly from the generalized Muirhead inequality \cref{eq:GeneralizedMuirhead} by setting $b_{\alpb}^{(C^+, \betab)} = \frac{c_{\alpb}^{(C^+, \betab)}}{\lam_{\alpb}^{(C^+, \betab)}}$ for all $\alpb \in C^+$, since
	\begin{align*}
		\sum_{\sig \in S_n} f^{\sig}_{(C^+, \betab)}(\xb) \ &= \ 
		\sum_{\sigma \in S_n} \sum_{\alpb \in C^+} c_{\alpb}^{(C^+, \betab)} \xb^{\sig(\alpb)} + c_{\betab}^{(C^+, \betab)} \xb^{\sig(\betab)} \\
		\ &\ge \ \sum_{\sigma \in S_n} \sum_{\alpb \in C^+} b_{\alpb}^{(C^+, \betab)} \lam_{\alpb}^{(C^+, \betab)} \xb^{\sig(\alpb)} - \left| c_{\betab}^{(C^+, \betab)} \right| \xb^{\sig(\betab)} \\
		&\ge \ \sum_{\sigma \in S_n} \sum_{\alpb \in C^+} b_{\alpb}^{(C^+, \betab)} \lam_{\alpb}^{(C^+, \betab)} \xb^{\sig(\alpb)} - \prod_{\alpb \in C^+} \left(\frac{c_{\alpb}^{(C^+, \betab)}}{\lam_{\alpb}^{(C^+, \betab)}}\right)^{\lam_{\alpb}^{(C^+, \betab)}} \xb^{\sig(\betab)} 
		\ \ge \ 0 
	\end{align*}
	holds for every $(C^+, \betab) \in \cC(A^+, \betab)$ with $\betab \in A^-$.
\end{proof}

\cref{corollary:SymmetricSONC} shows that we can derive the nonnegativity of symmetric sums of circuit polynomials from Muirhead's inequality, analogous to deriving the nonnegativity of circuit polynomials from the (weighted) AM-GM inequality. 

Now, we prove that symmetric polynomials are contained in the SONC cone if and only if their nonnegativity can be proven via \cref{corollary:SymmetricSONC}.
This result can be seen as a simplified restatement of \cite[Theorem 4.1]{Moustrou:Naumann:Riener:Theobald:Verdure:Symmetry} for the symmetric group $S_n$. 
Here, we rely only on the generalized Muirhead inequality to prove the statement, which illustrates that extensions of nonnegativity certificates to symmetric SONCs are a natural consequence of this inequality and do not rely on further algebraic techniques.

\begin{theorem}
	Let $A = A^+ \cup A^- \subset \N^n$ and $f$ be a symmetric polynomial supported on $\Sigma(A^+) \cup \Sigma(A^-) \subset \N^n$. 
	Then $f \in \SONC_{\Sigma(A^+), \Sigma(A^-)}$ if and only if it admits a representation of the form 
	\begin{align}
		f(\xb) \ = \ \sum_{\sig \in S_n} \sum_{\betab \in A^-} \sum_{(C^+, \betab) \in \cC(A^+, \betab)} f^{\sig}_{(C^+, \betab)}(\xb) \ ,
		\label{eq:SymmRepresentation}
	\end{align}
	with $f^{\sig}_{(C^+, \betab)}$ as in \cref{eq:SymmetricSum} such that one of the equivalent conditions in \cref{corollary:SymmetricSONC} is satisfied.
	\label{theorem:MuirheadSONCEquivalence}
\end{theorem}

\begin{proof}
	Assume first that $f$ can be written as in \cref{eq:SymmRepresentation} and that Condition (\cref{eq:CircuitCondition}) holds. 
	Then the nonnegativity of $f$ follows from \cref{corollary:SymmetricSONC}.
	We only need to prove that $f$ is also contained in the SONC cone.
	It holds by assumption that $f$ is a sum of circuit polynomials $f^{\sig}_{(C^+, \betab)}(\xb) = \sum_{\alpb \in C^+} c_{\alpb}^{(C^+, \betab)} \xb^{\sig(\alpb)} + c_{\betab}^{(C^+, \betab)} \xb^{\sig(\betab)}$. 
	All of these circuit polynomials are nonnegative as a consequence of the fact that 
	\begin{align}
		\left| c_{\betab}^{(C^+, \betab)} \right| \ \le \ \prod_{\alpb \in C^+} \left(\frac{c_{\alpb}^{(C^+, \betab)}}{\lam_{\alpb}^{(C^+, \betab)}}\right)^{\lam_{\alpb}^{(C^+, \betab)}}
	\end{align}
	holds for all $\betab \in A^-$ and all $(C^+, \betab) \in \cC(A^+, \betab)$. 
	Only checking this condition for one permutation $\sigma \in S_n$ is sufficient, because the barycentric coordinates of $\sigma(\betab)$ with respect to $\set{\sigma(\alpb) \ : \ \alpb \in C^+}$ are independent of the permutation $\sigma \in S_n$, see \cref{eq:2}. 
	Thus, $f(\xb) = \sum_{\sig \in S_n} \sum_{\betab \in A^-} \sum_{(C^+, \betab) \in \cC(A^+, \betab)} f^{\sig}_{(C^+, \betab)}(\xb)$ is a sum of nonnegative circuit polynomials and contained in $\SONC_{\Sigma(A^+), \Sigma(A^-)}$.

	For the remaining direction of the proof, let $f$ be a symmetric polynomial in $\SONC_{\Sigma(A^+), \Sigma(A^-)}$. 
	Then $f$ has a cancellation-free representation as a sum of nonnegative circuit polynomials $f_{(C^+, \betab)}(\xb) = \sum_{\alpb \in C^+} c_{\alpb}^{(C^+, \betab)} \xb^{\alpb} + c_{\betab}^{(C^+, \betab)} \xb^{\betab}$ of the form 
	\begin{align}
		f(\xb) \ = \ \sum_{\betab \in \Sigma(A^-)} \sum_{(C^+, \betab) \in \cC(\Sigma(A^+), \betab)} f_{(C^+, \betab)}(\xb),
		\label{eq:SONCRepNoSym}
	\end{align}
	that is, a representation that does not require exponents outside of $\Sigma(A^+) \cup \Sigma(A^-)$, 
	by \cite[Theorem 5.6]{Wang:nonnegative}.

	If we can now show that $f$ also has a representation as in \cref{eq:SymmRepresentation}, then the claim follows.
	For this, we need to explicitly re-write \cref{eq:SONCRepNoSym} in terms of a sum over $\sigma \in S_n$.
	Note that a similar statement is proven in \cite[Theorem 3.1]{Moustrou:Naumann:Riener:Theobald:Verdure:Symmetry} in the language of AGEs.
	Assume without loss of generality that $\# \Sigma(\Vector{\gamma}) = n!$ for all $\Vector{\gamma} \in A = A^+ \cup A^-$, that is, each permutation of an exponent vector results in a new term in our polynomial. 
	We can make this assumption since we could otherwise just refrain from simplifying the polynomial by combining circuit polynomials with the same support and still get $n!$ many circuit polynomials.

	Since $f$ is already symmetric, we can equivalently rewrite \cref{eq:SONCRepNoSym} to 
	\begin{align*}
		f(\xb) \ &= \ \sum_{\betab \in \Sigma(A^-)} \sum_{(C^+, \betab) \in \cC(\Sigma(A^+), \betab)} \sum_{\alpb \in C^+} c_{\alpb}^{(C^+, \betab)} \xb^{\alpb} + c_{\betab}^{(C^+, \betab)} \xb^{\betab} \\
		&= \ \frac{1}{n!} \sum_{\sigma \in S_n} \left( \sum_{\betab \in \Sigma(A^-)} \sum_{(C^+, \betab) \in \cC(\Sigma(A^+), \betab)} \sum_{\alpb \in C^+} c_{\alpb}^{(C^+, \betab)} \xb^{\sigma(\alpb)} + c_{\betab}^{(C^+, \betab)} \xb^{\sigma(\betab)} \right).
	\end{align*}
	If we define $f^{\sig}_{(C^+, \betab)} = \frac{1}{n!} \left(\sum_{\alpb \in C^+} c_{\alpb}^{(C^+, \betab)} \xb^{\sigma(\alpb)} + c_{\betab}^{(C^+, \betab)} \xb^{\sigma(\betab)}\right)$, then all $f^{\sig}_{(C^+, \betab)}$ are nonnegative circuit polynomials since the original $f_{(C^+, \betab)}$ from decomposition \cref{eq:SONCRepNoSym} are nonnegative circuit polynomials and permutation of the exponents preserves this property, see \cref{eq:2}. 
	Since these circuit polynomials are nonnegative if and only if the circuit number condition 
	\begin{align*}
		\left| c_{\betab}^{(C^+, \betab)} \right| \ \le \ \prod_{\alpb \in C^+} \left(\frac{c_{\alpb}^{(C^+, \betab)}}{\lam_{\alpb}^{(C^+, \betab)}}\right)^{\lam_{\alpb}^{(C^+, \betab)}}
	\end{align*} 
	holds, Condition (\cref{eq:CircuitCondition}) in \cref{corollary:SymmetricSONC} is satisfied and the claim follows.
\end{proof}

As a direct consequence, we obtain the following result about the symmetric closure of the SONC cone. 

\begin{corollary}
	Let $f(\xb) = \sum_{\alpb \in A^+} c_{\alpb} \xb^{\alpb} + \sum_{\betab \in A^+} c_{\betab} \xb^{\betab}$ be an arbitrary polynomial for some finite support set $A = A^+ \cup A^- \subset \N^n$. 
	Then $f$ is in $\signedSONC$ if and only if the corresponding symmetric polynomial 
	\begin{align}
		f_{\text{sym}}(\xb) \ = \ \sum_{\sigma \in S_n} \left(\sum_{\alpb \in A^+} c_{\alpb} \xb^{\sigma(\alpb)} + \sum_{\betab \in A^+} c_{\betab} \xb^{\sigma(\betab)} \right)
		\label{eq:SymmetrizedSONC}
	\end{align}
	is a SONC polynomial supported on $\Sigma(A^+) \ \cup \ \Sigma(A^-)$.
	\label{cor:SymmetricClosure}
\end{corollary}

\begin{proof}
	This follows immediately from \cref{theorem:MuirheadSONCEquivalence}.
\end{proof}

\begin{example}
	Consider the polynomial 
	\begin{align*}
		f(\xb) \ = \ \frac{1}{2}x_1^4 + \frac{1}{2}x_2^4x_3^4 + \frac{1}{4}x_2^4x_3^8 - x_1x_2x_3 - x_1x_2^2x_3^3 + \frac{3}{4}
	\end{align*}
	with support $A^+ = \set{\Matrix{0\\0\\0}, \Matrix{4\\0\\0}, \Matrix{0\\4\\4}, \Matrix{0\\4\\8}}$ and $A^- = \set{\Matrix{1\\1\\1}, \Matrix{1\\2\\3}}$. 
	We show first that $f$ is contained in $\signedSONC$.
	For this, note that we can write $f$ as a sum of the circuit polynomials
	\begin{align*}
		f_{\smallMatrix{1\\1\\1}} \ &= \ \frac{1}{4}x_1^4 + \frac{1}{4}x_2^4x_3^4 - x_1x_2x_3 + \frac{1}{2} \ , \text{ and } \\
		f_{\smallMatrix{1\\2\\3}} \ &= \ \frac{1}{4}x_1^4 + \frac{1}{4}x_2^4x_3^4 + \frac{1}{4}x_2^4x_3^8 - x_1x_2^2x_3^3 + \frac{1}{4} \ .
	\end{align*}
	Both of these circuit polynomials are nonnegative since they each satisfy $c_{\alpb}^{\left(\betab\right)} = \lam_{\alpb}^{\left(\betab\right)}$ for all $\alpb \in A^+$ and thus the circuit number condition becomes
	\begin{align*}
		|c_{\betab}| \ = \ 1 \ \le \ \prod_{\alpb \in A^+} \left(\frac{c_{\alpb}^{\left(\betab\right)}}{\lam_{\alpb}^{\left(\betab\right)}}  \right) \ = \ \prod_{\alpb \in A^+} \left(\frac{\lam_{\alpb}^{\left(\betab\right)}}{\lam_{\alpb}^{\left(\betab\right)}}  \right) \ = \ 1 \quad \text{ for all } \betab \in A^- \ .
	\end{align*}
	By \cref{cor:SymmetricClosure}, this immediately yields that the symmetrized version of $f$ given by 
	\begin{align*}
		f_{\text{sym}}(\xb) \ &= \ \frac{1}{2}x_1^4 + \frac{1}{2}x_2^4 + \frac{1}{2}x_3^4 + \frac{1}{2}x_2^4x_3^4 + \frac{1}{2}x_1^4x_2^4 + \frac{1}{2}x_1^4x_3^4 + \frac{1}{4}x_2^4x_3^8 + \frac{1}{4}x_1^4x_2^8 + \frac{1}{4}x_1^4x_3^8 \\
		&- \ x_1x_2x_3 - x_1x_2^2x_3^3 - x_1^2x_2x_3^3 - x_1x_2^3x_3^2 - x_1^3x_2^2x_3 - x_1^2x_2^3x_3 - x_1^3x_2x_3^2 + \frac{3}{4}
	\end{align*}
	is also a nonnegative polynomial in the SONC cone.
\end{example}

To illustrate that our results can be seen as a somewhat simplified version of those presented in \cite[Section 4]{Moustrou:Naumann:Riener:Theobald:Verdure:Symmetry}, we restate \cite[Example 5.1]{Moustrou:Naumann:Riener:Theobald:Verdure:Symmetry} to match our terminology.

\begin{example}[{\cite[Example 5.1]{Moustrou:Naumann:Riener:Theobald:Verdure:Symmetry}}]
	Consider the support set $A = A^+ \cup A^-$, where 
	\begin{align*}
		A^+ \ = \ \set{\Matrix{0\\0\\0}, \Matrix{7\\0\\0}, \Matrix{0\\7\\0}, \Matrix{0\\0\\7}} \ , \text{ and } \ 
		A^- \ = \ \set{\Matrix{1\\1\\2}, \Matrix{1\\2\\1}, \Matrix{2\\1\\1}, \Matrix{2\\2\\2}} \ .
	\end{align*}
	In order to satisfy the condition $A^+ \subset (2\N)^n$, or equivalently $\xb \in \R^n_{\ge 0}$, let $f$ be a symmetric polynomial on $A$ in the variables $\yb = (y_1, \ldots, y_n) = (x_1^2, \ldots, x_n^2)$. 
	Then we can write $f$ as 
	\begin{align}
		\begin{split}
			f(\yb) \ &= \ \sum_{\sig \in S_3} \left[ \left( c_{(000)}^{(1)} + c_{(700)}^{(1)} \yb^{\sig(700)} + c_{(070)}^{(1)} \yb^{\sig(070)} + c_{(007)}^{(1)} \yb^{\sig(007)} + c_{(112)} \yb^{\sig(112)}\right) \right. \\
			& \qquad \ \ + \left. \left( c_{(000)}^{(2)} + c_{(700)}^{(2)} \yb^{\sig(700)} + c_{(070)}^{(2)} \yb^{\sig(070)} + c_{(007)}^{(2)} \yb^{\sig(007)} + c_{(222)} \yb^{\sig(222)} \right)\right]
		\end{split}
	\label{eq:ExampleSymmetricSONC}
	\end{align}
	for some coefficient vector $\Vector{c} \in \R^A$ satisfying $c_{\alpb}^{(1)} + c_{\alpb}^{(2)} = c_{\alpb}$ for all exponents $\alpb \in A^+$. 
	\cref{corollary:SymmetricSONC} now tells us that instead of verifying the circuit number condition \cref{eq:CircuitCondition} for each of these circuit functions, it is sufficient to check them for one arbitrary permutation $\sig \in S_3$. 
	I.e., we would usually need to check the nonnegativity of four circuit polynomials ($12$ if we do not simplify the representation in \cref{eq:ExampleSymmetricSONC} by combining circuit functions with equal support). 
	By exploiting symmetries, we only need to compute two circuit numbers.
\end{example}

\section{Outlook}

First, we point out that the results presented in this article can be restated in the language of AGE functions without much difficulty, if this notation is preferred.
Following the ``SONC = SAGE'' result in \cite{Wang:nonnegative,Murray:Chandrasekaran:Wierman:NewtonPolytopes} it is clear that an analogue statement has to hold.

An interesting line of future research would be adapting our statements for the \struc{DSONC cone} -- a subcone of the SONC cone that was recently introduced in \cite{Heuer:deWolff:DualityOfSONC}. 
The DSONC cone is derived from the dual SONC cone, which has been studied in \cite{Dressler:Naumann:Theobald:DualSONC,Dressler:Heuer:Naumann:deWolff:DualSONCLP}, and inherits numerous structural properties of the original SONC cone.

We close the article with a few words on \struc{closure operators} applied to the SONC cone $\ASONC$. The cone of nonnegative polynomials is closed under various operations such as multiplication of two polynomials or a transformation of variables. The same holds for sums of squares, that is, the product of two sums of squares is a sum of squares, and if $f(\xb)$ is a sums of squares, the same will hold for $f(A\xb)$ for a matrix $A$. However, for the SONC cone, the situation is different: it is not closed under multiplication \cite{Dressler:Iliman:deWolff:Positivstellensatz}, and it is not closed under transformation of variables \cite{Dressler:Kurpisz:deWolff:Hypercube,Kurpisz:deWolff:Hierarchies}.
Indeed, this is not a shortcoming, but an opportunity to obtain generalized certificates of nonnegativity of the form $\cO(\ASONC)$ given by the closure of a suitable operator $\cO$ applied to the SONC cone $\ASONC$ (or another suitable cone of certificates of nonnegativity).
The very same approach was carried out by Ahmadi and Hall in \cite{Ahmadi:Hall:BasisPursuit} in the context of variable transformations applied to SDSOS certificates, which are particular SOS polynomials that happen to be also SONC; see also \cite{Ahmadi:Majumdar:DSOSandSDSOS,Kurpisz:deWolff:Hierarchies}.

The results of this paper can be seen as a contribution in this broader context of applying closure operators to suitable cones of certificates of nonnegativity: If we define an operator $\struc{\Sigma(\cdot)}$ for \struc{symmetric closure} of a cone (or, more generally, a set) of certificates of nonnegativity, that is, in the context of SONC
\begin{align*}
	\struc{\Sigma(\signedSONC)} \ := \ \set{f_{\text{sym}}(\xb) \text{ as in \cref{eq:SymmetrizedSONC}} \ : \ f(\xb) \in \signedSONC} \ ,
\end{align*}
then \cref{cor:SymmetricClosure} can be stated as
\begin{align*}
	\SONC_{\Sigma(A^+), \Sigma(A^-)} \ = \ \Sigma(\signedSONC) \ .
\end{align*}

\bibliographystyle{amsalpha}
\bibliography{gen_muirhead}

\end{document}